\documentclass[sn-mathphys]{sn-jnl}% Math and Physical Sciences Reference Style
\jyear{2021}%

\theoremstyle{thmstyleone}%
\newtheorem{theorem}{Theorem}%  meant for continuous numbers
\usepackage[euler]{textgreek}
\usepackage{amscd,amssymb,amsmath,amsthm}
\usepackage{graphicx}
\usepackage{caption}    % Umumiy caption uchun
\usepackage{subcaption} % Yonma-yon joylashtirish va alohida caption uchun

\usepackage{xcolor}

\usepackage{epstopdf}
\theoremstyle{thmstyletwo}%
\newtheorem{remark}{Remark}%
\newtheorem{lemma}{Lemma}

\theoremstyle{thmstylethree}%

\begin{document}

\title[Gibbs measure for mixed spins and mixed types model]{Gibbs measure for mixed spins and mixed types model}


\author[1,2]{\fnm{Muzaffar} \sur{Rahmatullaev}}\email{mrahmatullaev@rambler.ru}

\author*[2]{\fnm{Akbarkhuja} \sur{Tukhtabaev}}\email{akbarxoja.toxtaboyev@mail.ru}
\equalcont{These authors contributed equally to this work.}

\affil[1]{\orgname{Institute of Mathematics, Academy of Science}, \orgaddress{\street{9, University str}, \city{Tashkent}, \postcode{100174}, \country{Uzbekistan}}}

\affil*[2]{\orgname{Namangan State University}, \orgaddress{\street{161, Boburshoh street}, \city{Namangan}, \postcode{160107}, \country{Uzbekistan}}}


\abstract{In the present paper, we study the $(2,q)$-Ising-Potts model on the Cayley tree. We have derived a recurrence equation that shows the existence of a splitting Gibbs measure for this model. Furthermore, we have proven that for the $(2,q)$-Ising-Potts model on the Cayley tree of order $k\geq2$, there are at least 3 translation-invariant splitting Gibbs measures. We also prove that for the $(2,3)$-Ising-Potts model on the Cayley tree, specifically the binary tree, under certain conditions, there are at least 8 translation-invariant splitting Gibbs measures.}

\keywords{Ising-Potts model; Gibbs measure; translation-invariant; phase transition.}



\maketitle

\section{Introduction}\label{sec1}

Currently, special attention is paid to the study of mixed-type models. The results obtained for mixed-type models can be applied to each model individually, or to a combination of them, in order to represent the different states of components in a binary mixture (e.g., solid or liquid). By adjusting parameters within the model, such as the interaction energies between molecules, researchers can simulate the phase behavior of the mixture and predict meltig points, eutectic points, and other phase transitions.

In \cite{S16}, \cite{RR21} the Potts-SOS model, in \cite{Mc89}, \cite{RI24}, \cite{HOR25} the Ising-Potts model were studied. For a deeper understanding of the physical motivation behind the Ising-Potts model, refer to \cite{Mc89}.

In the present work, we consider the Ising-Potts model with parameter $\alpha$ and non-homogenous parameters on the Cayley tree. The model under study differs from the one considered in \cite{Mc89} and \cite{HOR25}. It is almost identical to the model in \cite{RI24}, but in \cite{RI24}, the spin values of the mixed model were taken to be the same for both the Ising and Potts components. In our case, however, the spin values for the Ising and Potts models are taken from different sets. The main goal of the work is to study the translation-invariant Gibbs measures for this model.

The structure of the paper is as follows. Section 2 provides preliminary information about Cayley tree, Gibbs measure for the Ising-Potts model.  In Section 2, we also find the functional-equation guaranteeing existence of a unique splitting Gibbs measure. Section 3 is devoted to the study of translation-invariant Gibbs measures for the Ising-Potts model on a Cayley tree. In section 4, we give the list of open problems related to for this model.

\section{Preliminaries}

\subsection{Cayley tree}
The Cayley tree $\Gamma ^{k}$ of order $k\geq1$ is an infinite regular tree, i.e. a graph each vertex originate exactly $k+1$ edges. Denote by $V$ and $L$ the set of vertices and edges of the Cayley tree $\Gamma ^{k}$, respectively. Two vertices $x$ and $y$ are called
\emph{nearest neighbours} if there exist an edge $l \in L$
connecting them and denote by $l = \langle x, y\rangle.$

Fix $x_0 \in {\Gamma ^k}$. For ${x,y}\in{\Gamma^k}$, denote by $d(x,y)$ the number of edges
in the shortest path connecting $x$ and $y$.

We set
$$
W_n = \{ x \in V: d(x, x_0) = n\} ,\ \ {V_n} = \{ x \in V: d(x, x_0) \le n\},\\
$$$${L_n} = \{ l = \langle x,y \rangle  \in L: x,y \in {V_n}\}.
$$

Let $x\in W_n$,
$$S(x)=\{y \in W_{n+1}: d(y,x)=1\}, \ \ {S_1}(x) = \{ y \in
V: d(y,x) = 1\},\ x\downarrow={S_1}(x)\setminus S(x).
$$
The set $S(x)$ is called the set of the direct successors of the
vertex $x$ (see \cite{RGM}).

\subsection{Gibbs measures for the $(2,q)$-Ising-Potts model}

Let $\Phi=\{-1,1\}$, $\Psi=\{1,2,...,q\}$ ($q\geq3$) and is assigned to the vertices of the tree
$\Gamma^k=(V,L)$. A configuration $(\sigma, s)$ on $A\subset V$ is then defined as
a function $x\in A\mapsto(\sigma(x),s(x))\in\Phi\times\Psi$. The set of all configurations on $A$
 is denoted by $\Omega_A  = (\Phi\times\Psi)^A$. One can see that
$\Omega_{V_n}=\Omega_{V_{n-1}}\times\Omega_{W_n}$. Using this, for given
configurations $(\sigma_{n-1},s_{n-1})\in\Omega_{V_{n-1}}$ and $(\sigma_{[n]}, s_{[n]})\in\Omega_{W_{n}}$ we
define their concatenations  by

\begin{equation}\label{unionconf}
((\sigma_{n-1}, s_{n-1}) \vee (\sigma_{[n]}, s_{[n]}))(x) = \left\{
\begin{array}{ll}
(\sigma_{n-1}, s_{n-1})(x), & \mbox{if}\,\,x \in {V_{n - 1}},\\
(\sigma_{[n]}, s_{[n]})(x), & \mbox{if}\,\,x \in {W_n}.
\end{array} \right.
\end{equation}
where $(\sigma, s)(x):=(\sigma(x), s(x))$. It is clear that $(\sigma_{n-1}, s_{n-1}) \vee (\sigma_{[n]}, s_{[n]})\in \Omega_{V_n}$.

The Hamiltonian of the $(2,q)$-Ising-Potts model is given by
\begin{equation}\label{Ham}
H_n(\sigma_n, s_n)=-\alpha \sum_{\langle x,y\rangle \in L_n}J_{I}(x,y)\sigma(x)\sigma(y)-(1-\alpha)\sum_{\langle x,y\rangle \in L_n}J_{P}(x,y)\delta_{s(x)s(y)},
\end{equation}
where $\sigma\in\Phi$, $s\in\Psi$, $\alpha\in[0,1]$.

\begin{remark}\label{comp1} Choosing by parameters in \eqref{Ham}, we get the following special cases
\begin{enumerate}
\item If $\alpha=0$ or  $J_{I}(x,y)=0$ and $J_{P}(x,y)=Const$ then \eqref{Ham} coincides with "ordinary" Potts model;\\[2mm]
 \item If $\alpha=1$ or  $J_{P}(x,y)=0$ and $J_{I}(x,y)=Const$ then \eqref{Ham} coincides with "ordinary" Ising model;\\[2mm]
 \item If $\alpha=0$ or  $J_{I}(x,y)=0$ and $J_{P}(x,y)=J_{P}+\frac{\alpha}{k+1}\left(\delta_{1\sigma(x)}+\delta_{1\sigma(y)}\right)$ then \eqref{Ham} coincides with Potts model with external field;\\[2mm]
  \item If $\alpha=1$ or  $J_{P}(x,y)=0$ and $J_{I}(x,y)=J_{I}+\frac{\alpha}{k+1}\left(\sigma(x)+\sigma(y)\right)$ then \eqref{Ham} coincides with Ising model with external field.\\[2mm]
\end{enumerate}
\end{remark}
Let us consider a probability measure
$\mu_{n}$ on $\Omega_{V_n}$ defined by
\begin{equation}\label{mu}
\mu_{n}(\sigma_n, s_n)=Z_{n}^{-1}\exp\left\{-\beta H_n(\sigma_n,s_n)+\sum _{x\in
W_n}h_{\sigma(x),s(x),x}\right\},
\end{equation}
where $\beta=\frac{1}{T}, T>0$- temperature,

$\textbf{h}_x=\{\left(h_{1,1,x},...,h_{1,q,x},h_{-1,1,x},...,h_{-1,q,x}\right)\in\mathbb R^{2q},x\in V\}$ is a collection of vectors,
 $Z_{n}^{-1}$ is the normalizing factor given by
\begin{equation}\label{ZN1}
Z_{n}=\sum_{(\sigma,s)\in\Omega_{V_n}}\exp\left\{-\beta H_n(\sigma_n,s_n)+\sum _{x\in
W_n}h_{\sigma(x),s(x),x}\right\}.
\end{equation}

A probability distribution
$\mu _{n}$ is said to be consistent if for all $n\geq1$ and
$(\sigma _{n - 1},s_{n-1}) \in {\Omega _{{V_{n - 1}}}}$, we have
\begin{equation}\label{comcon}
\sum\limits_{(\sigma_{[n]}, s_{[n]})  \in {\Omega _{{W_n}}}} {\mu_{n}((\sigma_{n-1}, s_{n-1}) \vee (\sigma_{[n]}, s_{[n]})) = \mu _{n -
1}(\sigma_{n-1}, s_{n-1})}.
\end{equation}

By the Kolmogorov theorem \cite{Athreya}, \cite{Haydarov}, there exists a unique \emph{splitting Gibbs measure} ${\mu}$ on the set $\Omega_V$
such that, for all $n$ and $(\sigma _{n},s_n) \in {\Omega _{{V_{n}}}}$, $${\mu}\left( {\left\{ {(\sigma, s)\mid_{V_n}=(\sigma_n,s_n)} \right\}} \right) = \mu_{n}(\sigma_n, s_n). $$

Denote  $\theta_{I}(x,y):=\exp\{\beta J_{I}(x,y)\}$, $\theta_{P}(x,y):=\exp\{\beta J_{P}(x,y)\}$, $z_{i,j,x}:=\exp\{h_{i,j,x}-h_{-1,q,x}\}$, $i=-1,1, j=1,2,...,q$.
To simplify the notation, we will write the above as follows $\theta_{I}(x,y)=\theta_{I}$, $\theta_{P}(x,y)=\theta_{P}$. We specifically highlight the constant case.

The following theorem is equivalent to the compatibility condition and ensures the existence of Gibbs measures.
\begin{theorem}\label{comp.eq} The measures $\mu_{n}(\sigma_n, s_n), n=1,2,...,$ satisfy the compatibility condition \eqref{comcon} if and only if for any $x\in V\setminus\{x_0\}$ the following equation holds:

If $(i,j)\neq(-1,q)$ then
$$z_{i,j,x}=$$
\begin{equation}\label{funceq}
\prod\limits_{y\in S(x)}\frac{\theta_{I}^{\alpha (1-i)}\left(\sum\limits_{v=1}^{q}\left(\theta_{I}^{2\alpha i}z_{1,v,y}+z_{-1,v,y} \right)+\left(\theta_{P}^{1-\alpha}-1\right)\left(\theta_{I}^{2\alpha i}z_{1,j,y}+z_{-1,j,y}\right)\right) }{\sum\limits_{v=1}^{q-1}\left(z_{1,v,y}+\theta_{I}^{2\alpha }(x,y)z_{-1,v,y} \right)+\theta_{P}^{1-\alpha}\left(z_{1,q,y}+\theta_{I}^{2\alpha }(x,y)\right)}\end{equation}

If $(i,j)=(-1,q)$ then  $z_{-1,q,x}=1$.
\end{theorem}

\begin{proof}
\emph{Necessity.} We prove that \eqref{comcon} $\Rightarrow$ \eqref{funceq}.

We first compute some necessary calculation.

1) We rewrite \eqref{Ham}
\begin{eqnarray*}
H_n(\sigma_n, s_n)&=&-\alpha \sum_{\langle x,y\rangle \in L_n}J_{I}(x,y)\sigma(x)\sigma(y)-(1-\alpha)\sum_{\langle x,y\rangle \in L_n}J_{P}(x,y)\delta_{s(x)s(y)}\\[2mm]
&=&-\alpha \sum_{\langle x,y\rangle \in L_{n-1}}J_{I}(x,y)\sigma(x)\sigma(y)
-\alpha\sum_{x\in W_{n-1}}\sum_{y\in S(x)}J_{I}(x,y)\sigma(x)\sigma(y)\\[2mm]
&-&(1-\alpha) \sum_{\langle x,y\rangle \in L_{n-1}}J_{P}(x,y)\delta_{s(x)s(y)}\\[2mm]
&-&(1-\alpha)\sum_{x\in W_{n-1}}\sum_{y\in S(x)}J_{P}(x,y)\delta_{s(x)s(y)}\\[2mm]
&=&H_{n-1}(\sigma_{n-1}, s_{n-1})-\alpha\sum_{x\in W_{n-1}}\sum_{y\in S(x)}J_{I}(x,y)\sigma_{n-1}(x)\sigma_{[n]}(y)\\[2mm]
&-&(1-\alpha)\sum_{x\in W_{n-1}}\sum_{y\in S(x)}J_{P}(x,y)\delta_{s_{n-1}(x)s_{[n]}(y)},\\[2mm]
\end{eqnarray*}
i.e.

$$H_n(\sigma_n, s_n)=H_{n-1}(\sigma_{n-1}, s_{n-1})-\alpha\sum_{x\in W_{n-1}}\sum_{y\in S(x)}J_{I}(x,y)\sigma_{n-1}(x)\sigma_{[n]}(y)$$
\begin{equation}\label{rH}
-(1-\alpha)\sum_{x\in W_{n-1}}\sum_{y\in S(x)}J_{P}(x,y)\delta_{s_{n-1}(x)s_{[n]}(y)}.
\end{equation}
2) We simplify the following expression
\begin{equation}\label{nthg}
\sum _{x\in W_n}h_{\sigma(x),s(x),x}=\sum_{x\in W_{n-1}}\sum_{y\in S(x)}h_{\sigma(y),s(y),y}=\sum_{x\in W_{n-1}}\sum_{y\in S(x)}h_{\sigma_{[n]}(y),s_{[n]}(y),y}.
\end{equation}
3) Putting \eqref{rH} and \eqref{nthg} into \eqref{mu}, we have
\begin{eqnarray*}
\mu_{n}(\sigma_n, s_n)&=&Z_{n}^{-1}\exp\left\{-\beta H_n(\sigma_n,s_n)+\sum _{x\in
W_n}h_{\sigma(x),s(x),x}\right\}\\[2mm]
\end{eqnarray*}
$$=Z_n^{-1}\exp\{-\beta H_{n-1}(\sigma_{n-1},s_{n-1})+$$ $$+\sum_{x\in W_{n-1}}\sum_{y\in S(x)}\left[\beta \alpha  J_{I}(x,y)\sigma_{n-1}(x)\sigma_{[n]}(y)
+\beta (1-\alpha) J_{P}(x,y)\delta_{s_{n-1}(x)s_{[n]}(y)}\right]+$$ $$+\sum_{x\in W_{n-1}}\sum_{y\in S(x)}\left[ h_{\sigma_{[n]}(y),s_{[n]}(y),y}\right]\}.$$

$$\mu_{n}(\sigma_n, s_n)=Z_n^{-1}\exp\{-\beta H_{n-1}(\sigma_{n-1},s_{n-1})\}\cdot$$

$$\cdot\prod_{x\in W_{n-1}}\prod_{y\in S(x)}\exp\{\left[\beta \alpha  J_{I}(x,y)\sigma_{n-1}(x)\sigma_{[n]}(y)
+\beta (1-\alpha) J_{P}(x,y)\delta_{s_{n-1}(x)s_{[n]}(y)}\right]\}$$
\begin{equation}\label{rM}\cdot\prod_{x\in W_{n-1}}\prod_{y\in S(x)}\exp\{ h_{\sigma_{[n]}(y),s_{[n]}(y),y}\}.
\end{equation}
4) \eqref{mu} follows that

$$\mu_{n}(\sigma_{n-1}, s_{n-1})=Z_{n-1}^{-1}\exp\{-\beta H_{n-1}(\sigma_{n-1},s_{n-1})\}$$
\begin{equation}\label{n1}\cdot\prod_{x\in
W_{n-1}}\exp\{h_{\sigma_{n-1}(x),s_{n-1}(x),x}\}.
\end{equation}
5) \eqref{comcon} yields that

 $$\mu _{n -1}(\sigma_{n-1}, s_{n-1})=\sum\limits_{(\sigma_{[n]}, s_{[n]})\in \Omega _{W_n}} \mu
_{n}((\sigma_{n-1}, s_{n-1}) \vee (\sigma_{[n]}, s_{[n]}))$$$$=\sum\limits_{(\sigma_{[n]}, s_{[n]})  \in {\Omega _{W_n}}} \mu
_{n}(\sigma_n,s_n).$$

Putting \eqref{n1} and \eqref{rM} into the last equation, we get
$$\prod _{x\in
W_{n-1}}\exp\{h_{\sigma_{n-1}(x),s_{n-1}(x),x}\}=$$
$$=\frac{Z_{n-1}}{Z_n}\sum\limits_{(\sigma_{[n]}, s_{[n]})  \in {\Omega _{W_n}}} \prod_{x\in W_{n-1}}\prod_{y\in S(x)}\exp[\beta \alpha  J_{I}(x,y)\sigma_{n-1}(x)\sigma_{[n]}(y)$$
\begin{equation}\label{cc3}
+\beta (1-\alpha) J_{P}(x,y)\delta_{s_{n-1}(x)s_{[n]}(y)}+ h_{\sigma_{[n]}(y),s_{[n]}(y),y}].
\end{equation}

Fix $x\in W_{n-1}$, rewrite the last equality for $(\sigma_{n-1},s_{n-1})=(i,j)\in\ \Phi\times\Psi$ and $(\sigma_{n-1},s_{n-1})=(-1,q)$, then dividing each of them to the last one, we have
$$
\prod\limits_{y\in S(x)}\frac{\sum\limits_{(u,v)\in \Phi\times\Psi}\exp\left[\beta \alpha  J_{I}(x,y)iu
+\beta (1-\alpha) J_{P}(x,y)\delta_{jv}+ h_{u,v,y}\right]}{\sum\limits_{(u,v)\in \Phi\times\Psi}\exp\left[-\beta \alpha  J_{I}(x,y)u
+\beta (1-\alpha) J_{P}(x,y)\delta_{qv}+ h_{u,v,y}\right]}$$
\begin{equation}\label{mainone}=\exp\{h_{i,j,x}-h_{-1,q,x}\},
\end{equation}
where $(i,j)\in\Phi\times\Psi$.

Using notations above $\theta_{I}(x,y)=\theta_{I}=\exp\{\beta J_{I}(x,y)\}$, $\theta_{P}(x,y)=\theta_{P}=\exp\{\beta J_{P}(x,y)\}$, the fraction LHS of \eqref{mainone} equals
$$\frac{\sum\limits_{v\in \Psi}\theta_{P}^{(1-\alpha)\delta_{jv}}\theta_{I}^{-\alpha i}\left(\theta_{I}^{2\alpha i}\exp\{h_{1,v,y}\}+\exp\{h_{-1,v,y}\} \right)}{\sum\limits_{v\in \Psi}\theta_{P}^{(1-\alpha)\delta_{qv} }\theta_{I}^{-\alpha }\left( \exp\{ h_{1,v,y}\}+ \theta_{I}^{2\alpha }\exp\{h_{-1,v,y}\} \right)}
=$$
$$=\frac{\theta_{I}^{-\alpha i}\left(\sum\limits_{v\neq j}\left(\theta_{I}^{2\alpha i}e^{h_{1,v,y}}+e^{h_{-1,v,y}} \right)+\theta_{P}^{1-\alpha} \left(\theta_{I}^{2\alpha i}e^{h_{1,j,y}}+e^{h_{-1,j,y}}\right)\right)}{\theta_{I}^{-\alpha}\left(\sum\limits_{v\neq q}\left(e^{h_{1,v,y}}+\theta_{I}^{2\alpha}e^{h_{-1,v,y}} \right)+\theta_{P}^{1-\alpha}\left(e^{h_{1,q,y}}+\theta_{I}^{2\alpha}e^{h_{-1,q,y}}\right)\right) }$$

Again using notation above $z_{i,j,x}=\exp\{h_{i,j,x}-h_{-1,q,y}\}$, $i=-1,1, j=1,2,...,q$, we have  $z_{-1,q,x}=1$ and
\eqref{mainone} follows that
$$z_{i,j,x}=$$
$$
\prod\limits_{y\in S(x)}\frac{\theta_{I}^{\alpha (1-i)}\left(\sum\limits_{v=1}^{q}\left(\theta_{I}^{2\alpha i}z_{1,v,y}+z_{-1,v,y} \right)+\left(\theta_{P}^{1-\alpha}-1\right)\left(\theta_{I}^{2\alpha i}z_{1,j,y}+z_{-1,j,y}\right)\right) }{\sum\limits_{v=1}^{q-1}\left(z_{1,v,y}+\theta_{I}^{2\alpha }z_{-1,v,y} \right)+\theta_{P}^{1-\alpha}\left(z_{1,q,y}+\theta_{I}^{2\alpha }\right)}.
$$

\emph{Sufficiency.} We prove that \eqref{funceq} $\Rightarrow$ \eqref{comcon}. Suppose that \eqref{funceq} holds. It is equivalent to the representation
$$\prod\limits_{y\in S(x)}\sum\limits_{(u,v)\in \Phi\times\Psi}\exp\left[\beta \alpha  J_{I}(x,y)iu
+\beta (1-\alpha) J_{P}(x,y)\delta_{jv}+ h_{u,v,y}\right]=$$
\begin{equation}\label{cc9}
a(x)\exp\{h_{i,j,x}\}
\end{equation}
for some function $a(x)>0, x\in V$.

We rewrite LHS of \eqref{comcon} as

$$\sum\limits_{(\sigma_{[n]}, s_{[n]})  \in {\Omega _{{W_n}}}} \mu_{n}((\sigma_{n-1}, s_{n-1}) \vee (\sigma_{[n]}, s_{[n]})) =Z_n^{-1}\exp\{-\beta H_{n-1}(\sigma_{n-1},s_{n-1})\}\cdot$$
\begin{equation}\label{cc12}
\cdot\prod_{x\in W_{n-1}}\prod_{y\in S(x)}\sum\limits_{(u,v)\in \Phi\times\Psi}\exp\{\beta \alpha  J_{I}\sigma_{n-1}(x)u
+\beta (1-\alpha) J_{P}\delta_{s_{n-1}(x)v}+ h_{u,v,y}\}.
\end{equation}

Substituting \eqref{cc9} into \eqref{cc12} and denoting $A_n(x)=\prod\limits_{x\in W_{n-1}}a(x)$, we get
$$\sum\limits_{(\sigma_{[n]}, s_{[n]})  \in {\Omega _{{W_n}}}} \mu_{n}((\sigma_{n-1}, s_{n-1}) \vee (\sigma_{[n]}, s_{[n]}))$$
\begin{equation}\label{cc13}
  = \frac{A_{n-1}}{Z_n}\exp\{-\beta H_{n-1}(\sigma_{n-1},s_{n-1})\} \prod\limits_{x\in W_{n-1}}\exp\{h_{\sigma_{n-1},s_{n-1},x}\}
\end{equation}
Since $\mu_{n}$ is a probability, we should have
$$\sum\limits_{(\sigma_{n-1}, s_{n-1})\in \Omega_{V_{n-1}}}\sum\limits_{(\sigma_{[n]}, s_{[n]})  \in {\Omega _{{W_n}}}} \mu_{n}((\sigma_{n-1}, s_{n-1}) \vee (\sigma_{[n]}, s_{[n]}))=1.$$
\eqref{cc13} yields
\begin{equation}\label{cc14}
\sum\limits_{(\sigma_{[n]}, s_{[n]})  \in {\Omega _{{W_n}}}} \mu_{n}((\sigma_{n-1}, s_{n-1}) \vee (\sigma_{[n]}, s_{[n]})) = \frac{A_{n-1}}{Z_n}\mu_{n-1}(\sigma_{n-1}, s_{n-1}) Z_{n-1}.
\end{equation}
or

$$1=\sum\limits_{(\sigma_{n-1}, s_{n-1})\in \Omega_{V_{n-1}}}\sum\limits_{(\sigma_{[n]}, s_{[n]})  \in {\Omega _{{W_n}}}} \mu_{n}((\sigma_{n-1}, s_{n-1}) \vee (\sigma_{[n]}, s_{[n]})),$$
\begin{equation}\label{cc15}
1=\sum\limits_{(\sigma_{n-1}, s_{n-1})\in \Omega_{V_{n-1}}}\frac{A_{n-1}}{Z_n}\mu_{n-1}(\sigma_{n-1}, s_{n-1}) Z_{n-1},
\end{equation}
or
\begin{equation}\label{cc16}
Z_n=A_{n-1}Z_{n-1}.
\end{equation}
Substituting \eqref{cc16} into \eqref{cc14}, we have

\begin{equation}\label{cc17}
\sum\limits_{(\sigma_{[n]}, s_{[n]})  \in {\Omega _{{W_n}}}} \mu_{n}((\sigma_{n-1}, s_{n-1}) \vee (\sigma_{[n]}, s_{[n]})) = \mu_{n-1}(\sigma_{n-1}, s_{n-1}).
\end{equation}
\end{proof}
\section{Translation-invariant Gibbs measures}
Let $\textbf{h}_x=\textbf{h}$ for all $x\in V$. The measure corresponding to $\mu_{\textbf{h}}$ is called \emph{translation-invariant} Gibbs measure (TIGM). Let $\theta_{I}(x,y)=\theta_{I}$, $\theta_{P}(x,y)=\theta_{P}$ for all $\langle x,y\rangle\in L_n$, i.e. they be constants.
According to \eqref{funceq} we have

\begin{equation}\label{funceqti}z_{i,j}=\left(\frac{\theta_I^{\alpha (1-i)}\left(\sum\limits_{v=1}^{q}\left(\theta_I^{2\alpha i}z_{1,v}+z_{-1,v} \right)+\left(\theta_P^{1-\alpha}-1\right)\left(\theta_I^{2\alpha i}z_{1,j}+z_{-1,j}\right) \right)}{\sum\limits_{v=1}^{q-1}\left(z_{1,v}+\theta_I^{2\alpha }z_{-1,v} \right)+\theta_P^{1-\alpha}\left(z_{1,q}+\theta_{I}^{2\alpha }\right)}\right)^k.
\end{equation}

Let $\textbf{z}=(z,z,...,z)$. Then \eqref{funceqti}  follows that

\begin{equation}\label{sys fa1}
\left\{
\begin{array}{llll}
\ z=\left(\frac{z(\theta_{I}^{\alpha }+\theta_{I}^{-\alpha })(\theta_{P}^{1-\alpha}+q-2)+\theta_I^{\alpha }z+\theta_I^{-\alpha }}{z(\theta_{I}^{\alpha }+\theta_{I}^{-\alpha })(q-1)+\theta_{P}^{1-\alpha}(\theta_I^{-\alpha }z+\theta_I^{\alpha })}\right)^k,\\[3mm]
\ z=\left(\frac{z(\theta_{I}^{\alpha }+\theta_{I}^{-\alpha })(\theta_{P}^{1-\alpha}+q-2)+\theta_I^{-\alpha }z+\theta_I^{\alpha }}{z(\theta_{I}^{\alpha }+\theta_{I}^{-\alpha })(q-1)+\theta_{P}^{1-\alpha}(\theta_I^{-\alpha }z+\theta_I^{\alpha })}\right)^k.\\[3mm]
\end{array}\right.
\end{equation}

Not complicated calculation shows that the system \eqref{sys fa1} has a unique $z=1$ solution.
\begin{remark}\label{trivial} \eqref{funceqti} has the solution $z=(1,1,...,1)$.
\end{remark}
Let operator $W:\mathbb R^{2q-1}\rightarrow  \mathbb R^{2q-1}$, i.e., $$W\left((z_{1,1}, z_{1,2},...,z_{1,q},z_{-1,1}, z_{-1,2},...,z_{-1,q-1})\right)=$$$$(z'_{1,1}, z'_{1,2},...,z'_{1,q},z'_{-1,1}, z'_{-1,2},...,z'_{-1,q-1})$$ defined as

$$z'_{i,j}=$$
\begin{equation}\label{funceqtiop}\left(\frac{\sum\limits_{v=1}^{q}\left(\theta_I^{\alpha i}z_{1,v}+\theta_{I}^{-\alpha i}z_{-1,v} \right)+\left(\theta_P^{1-\alpha}-1\right)\left(\theta_I^{\alpha i}z_{1,j}+\theta_{I}^{-\alpha i}z_{-1,j}\right) }{\sum\limits_{v=1}^{q-1}\left(\theta_I^{-\alpha}z_{1,v}+\theta_I^{\alpha }z_{-1,v} \right)+\theta_P^{1-\alpha}\left(\theta_{I}^{-\alpha }z_{1,q}+\theta_{I}^{\alpha }\right)}\right)^k.
\end{equation}

It is easy to see that the following sets are invariant respect to the operator $W$

$$I_1=\{\textbf{z}=(\underbrace{z,z,...,z,1}_q,z,z,...,z)\in\mathbb R^{2q-1}\},$$
$$I_2=\{\textbf{z}=(\underbrace{z_1,...,z_1,1}_q,z_2,...,z_2)\in\mathbb R^{2q-1}\}.$$

\textbf{Case 1.}  Now, we shall consider \eqref{funceqti} on the set $I_1$. Then we have
\begin{equation}\label{funceqtii1}
z=\left(\frac{z(\theta_{I}^{\alpha }+\theta_{I}^{-\alpha })(\theta_{P}^{1-\alpha}+q-2)+\theta_I^{\alpha }+\theta_I^{-\alpha }}{(\theta_{I}^{\alpha }+\theta_{I}^{-\alpha })\left(z(q-1)+\theta_{P}^{1-\alpha}\right)}\right)^k,
\end{equation}
or
\begin{equation}\label{funceqtii1}
z=\left(\frac{z(\theta_{P}^{1-\alpha}+q-2)+1}{z(q-1)+\theta_{P}^{1-\alpha}}\right)^k.
\end{equation}
Denote $x:=z(\theta_P^{1-\alpha}+q-2),\ A:=\frac{(q-1)^k}{\left(\theta_p^{1-\alpha}+q-2\right)^{k+1}},\ B:=\frac{\theta_p^{1-\alpha}(\theta_p^{1-\alpha}+q-2)}{q-1}$,\\ then \eqref{funceqtii1} yields that
\begin{equation}\label{funceqtii11}
Ax=\left(\frac{1+x}{B+x}\right)^k,
\end{equation}
here $A>0,\ B>0,\ k\in \mathbb N, x\geq0$.

Denote $$\theta_c=$$$$\theta_c(k,q,\alpha)=\left(\frac{\sqrt{(q-2)^2(k-1)^2+4(q-1)(k+1)^2}-(q-2)(k-1)}{2(k-1)}\right)^{\frac{1}{1-\alpha}}.$$

The presence of at least two distinct measures for the model indicates the occurrence of a \emph{phase transition}.

\eqref{funceqtii11} is well-studied in \cite{Preston}, using that results we have immediately the following theorem

\begin{theorem}\label{nti} The following assertions hold
 \begin{itemize}
   \item if $\theta_P\leq \theta_c$ or $k=1$ then the equation \eqref{funceqtii11} has a unique positive solution. Moreover, in this case, for the model \eqref{Ham} there does not exist a phase transition.
   \item if $\theta_P>\theta_c$ and $k>1$ then there exist $\eta_1(B,k)$, $\eta_2(B,k)$ with $0<\eta_1(B,k)<\eta_2(B,k)$ such that the equation \eqref{funceqtii11} has
   \begin{enumerate}
     \item three positive solutions if $A\in (\eta_1, \eta_2)$;
     \item two positive solutions if $A\in \{\eta_1, \eta_2\}$;
      \end{enumerate}
   In fact,
 \begin{equation}\label{xi}
 \eta_i=\eta_i(B,k)=\frac{1}{x_i}\left(\frac{1+x_i}{B+x_i}\right)^k,
 \end{equation}
 where $x_1,x_2$ are solutions of the following quadratic equation
 \end{itemize}
 $$x^2+[2-(B-1)(k-1)]x+B=0.$$
 Furthermore, in this case,  for the model \eqref{Ham} there exists a phase transition.
\end{theorem}

\textbf{Subcase $k=2$.}  In this subcase we shall consider the equation \eqref{funceqtii1} for $k=2$. Then we have

\begin{equation}\label{k21}
z=\left(\frac{z(\theta_{P}^{1-\alpha}+q-2)+1}{z(q-1)+\theta_{P}^{1-\alpha}}\right)^2.
\end{equation}
Using notation $b:=\theta_{P}^{1-\alpha}$, we rewrite \eqref{k21} as follows
\begin{equation}\label{k22}
(z-1)\left((q-1)^2z^2-\left((b-1)^2-2(q-1)\right)z+1\right)=0.
\end{equation}
The solutions of \eqref{k22} are
\begin{equation}\label{k23}
z_0=1,\ z_{1,2}=\frac{(b-1)^2-2(q-1)\pm\mid b-1\mid\sqrt{(b-1)^2-4(q-1)}}{2(q-1)^2}.
\end{equation}
\begin{theorem}\label{k2ti}
Let $k=2$. For the translation- invariant Gibbs measures of the $(2,q)$-Ising-Potts model, corresponding to the solutions of \eqref{funceqti} on the invariant set $I_1$, the following assertions hold
\begin{itemize}
  \item If $\theta_{P}>\left(1+2\sqrt{q-1}\right)^{\frac{1}{1-\alpha}}$ then there are 3 TIGMs;
  \item If $\theta_{P}=\left(1+2\sqrt{q-1}\right)^{\frac{1}{1-\alpha}}$ then there are 2 TIGMs;
  \item If $0<\theta_{P}<\left(1+2\sqrt{q-1}\right)^{\frac{1}{1-\alpha}}$ then there exists a unique TIGM.
\end{itemize}
\end{theorem}
\begin{proof}
Denote $$\Delta=(b-1)^2-4(q-1).$$
Clear that if $\Delta\geq 0$ then there exist $z_{1,2}$ and both are positive. The condition $\Delta\geq 0$ is equivalent to $\theta_{P}\geq\left(1+2\sqrt{q-1}\right)^{\frac{1}{1-\alpha}}.$ The rest of the proof is straightforward.
\end{proof}
\begin{remark}\label{com}
\begin{itemize}
  \item If $\alpha\rightarrow 1$, \eqref{Ham} drives Ising model. According to Theorem \ref{k2ti},  there is always at least one Gibbs measure for the Ising model on the Cayley tree of order two.
  \item If $\alpha\rightarrow 0$, \eqref{Ham} drives Potts model. Again due to Theorem \ref{k2ti},
  \begin{itemize}
    \item if $\theta_{P}>\left(1+2\sqrt{q-1}\right)$ then there exist at least three Gibbs measures,
    \item if $\theta_{P}=\left(1+2\sqrt{q-1}\right)$ then there exist at least two Gibbs measures,
    \item if $0<\theta_{P}<\left(1+2\sqrt{q-1}\right)$ then there exists at least one Gibbs measure
  \end{itemize}
for the Potts model on the Cayley tree of order two.
\end{itemize}
 \end{remark}

\begin{remark}\label{compare}
For $k=2$, without Theorem \ref{k2ti}, using Theorem \ref{nti}, we can find number of positive solutions of \eqref{k21}. However, using explicit form of solutions, we can later check  extremality of the measures.
\end{remark}
\textbf{Case 2.} In this case, we shall consider \eqref{funceqti} on the set $I_2$. For simplicity, let $k=2$.  The equation \eqref{funceqti} follows that
\begin{equation}\label{sys i21}
\left\{
\begin{array}{llll}
\ z_{1}=\left(\frac{(\theta_I^{\alpha}z_1+\theta_I^{-\alpha}z_2)(q+\theta_P^{1-\alpha}-2)+\theta_I^{\alpha}+\theta_I^{-\alpha}}{(\theta_I^{-\alpha}z_1+\theta_I^{\alpha}z_2)(q-1)+\theta_P^{1-\alpha}(\theta_I^{\alpha}+\theta_I^{-\alpha})}\right)^2,\\[3mm]
\ z_{2}=\left(\frac{(\theta_I^{-\alpha}z_1+\theta_I^{\alpha}z_2)(q+\theta_P^{1-\alpha}-2)+\theta_I^{\alpha}+\theta_I^{-\alpha}}{(\theta_I^{-\alpha}z_1+\theta_I^{\alpha}z_2)(q-1)+\theta_P^{1-\alpha}(\theta_I^{\alpha}+\theta_I^{-\alpha})}\right)^2.\\[3mm]
\end{array}\right.
\end{equation}

\begin{remark}\label{positive}
Let $(z_1,z_2)$ be any solution of \eqref{sys i21}. Since RHSs of \eqref{sys i21} are positive, both $z_1$ and $z_2$ are positive.
\end{remark}

Using the notation $u_1=\sqrt{z_1},\ u_2=\sqrt{z_2}$ then we have

\begin{equation}\label{sys i22}
\left\{
\begin{array}{llll}
\ u_{1}\left[(\theta_I^{-\alpha}u_1^2+\theta_I^{\alpha}u_2^2)(q-1)+\theta_P^{1-\alpha}(\theta_I^{\alpha}+\theta_I^{-\alpha})\right]=\\
(\theta_I^{\alpha}u_1^2+\theta_I^{-\alpha}u_2^2)(q+\theta_P^{1-\alpha}-2)+\theta_I^{\alpha}+\theta_I^{-\alpha},\\[3mm]
\ u_{2}\left[(\theta_I^{-\alpha}u_1^2+\theta_I^{\alpha}u_2^2)(q-1)+\theta_P^{1-\alpha}(\theta_I^{\alpha}+\theta_I^{-\alpha})\right]=\\
(\theta_I^{-\alpha}u_1^2+\theta_I^{\alpha}u_2^2)(q+\theta_P^{1-\alpha}-2)+\theta_I^{\alpha}+\theta_I^{-\alpha}.\\[3mm]
\end{array}\right.
\end{equation}
Denote $a=\theta_I^{\alpha}$, $b=\theta_P^{1-\alpha}$.
We find $u_2^2$ from the first equation of \eqref{sys i22}
\begin{equation}\label{u2}
u_2^2=\frac{-(q-1)u_1^3+a^2(q+b-2)u_1^2-(a^2+1)bu_1+u_1^2+1}{a^2(q-1)u_1-(q+b-2)}.
\end{equation}
Substituting \eqref{u2} into the second equation of \eqref{sys i22} and after some algebras we get
$$
(u_1-1)(a^2(q-1)u_1-(b+q-2))^2(-(q-1)u_1^2+(b-1)u_1-1)$$
\begin{equation}\label{sys i23}\cdot(A_1u_1^4+B_1u_1^3+C_1u_1^2+D_1u_1+E_1)=0,
\end{equation}
where $$A_1=A_1(a,b,q)=-(q-1)^2(b+q-2)^2(a^2-1)^2,$$ $$B_1=B_1(a,b,q)=(q-1)(b+q-2)^3(a^2-1)^3,$$
$$C_1=C_1(a,b,q)=-(b+q-2)(a^2-1)$$
$$\cdot[(a^2-1)b^3+((q-1)a^4+(2q-5)a^2-5q+8)b^2+$$
$$+((q^2-3q+2)a^4+(2q^2-9q+10)a^2-5q^2+18q-16)b+$$$$(q^3-4q^2+8q-6)a^2-q^3+6q^2-12q+8],$$
$$D_1=D_1(a,b,q)=(q+b-2)^2(a^2-1)^2$$$$[(a^2+1)b^2+(q-2)(a^2+1)b+a^2(a^2-3)(q-1)],$$
$$E_1=E_1(a,b,q)=-(a^2+1)b^4+(-a^4+(-2q+6)a^2-2q+3)b^3+$$$$((-q+4)a^4+(-q^2+12q-18)a^2-q^2+q+2)b^2-$$
$$-(q-2)((q-4)a^4-2(4q-7)a^2+3(q-2))b-(q-1)^2a^6+$$$$(-q^3+5q^2-10q+7)a^4+2(q-2)^3a^2-(q-2)^3.$$

The equation \eqref{sys i23} has $u_0^{(1)}=1$, $$u_{1,2}^{(1)}=\frac{b-1\pm\sqrt{(b-1)^2-4(q-1)}}{2(q-1)}$$ and $$u_3^{(1)}=\frac{b+q-2}{a^2(q-1)}$$ solutions.

Now, we consider the equation
\begin{equation}\label{sys i24}
f(u_1, a,b,q):=A_1u_1^4+B_1u_1^3+C_1u_1^2+D_1u_1+E_1=0.
\end{equation}

For simplicity, let we consider
\begin{equation}\label{aa3}
f(x,a,a,3)=0.
\end{equation}
($a=b$ follows that $J_I=\frac{\alpha}{1-\alpha}J_p$.) \eqref{aa3} is equivalent to

$$-4(a+1)^4x^4+2(a-1)(a+1)^6x^3-(a+1)^2(2a^5+5a^4+6a^3+6a^2+8a+1)x^2+$$\begin{equation}\label{aa31}+a(3a^2+4a-1)(a+1)^3x-((a+1)^5+a^2)=0.
\end{equation}

\begin{lemma}\label{numsol} Let $\mathcal N$ be a number of positive solutions of \eqref{aa31}. Then the following assertions hold
\begin{equation}\label{x}
\mathcal N=\left\{
	\begin{array}{ll}
		 0,\ \ \ \ \mbox{if} \ \ a\in (0, A^c_1); \\[2mm]
		 2,\ \ \ \ \mbox{if} \ \ a\in [A^c_1, A^c_2); \\[2mm]
 4,\ \ \ \ \mbox{if} \ \ a\in [A^c_2, +\infty),
	\end{array}\right.
\end{equation}
where $A^c_1\approx2,010\, \ A^c_2\approx4,921$.
\end{lemma}
\begin{proof}
Denote $$P(x)=-4(a+1)^4x^4+2(a-1)(a+1)^6x^3-(a+1)^2(2a^5+5a^4+6a^3+6a^2+8a+1)x^2+$$$$+a(3a^2+4a-1)(a+1)^3x-((a+1)^5+a^2),$$
or to shorten notation, we write
$$P(x)=A_1x^4+B_1x^3+C_1x^2+D_1x+E_1.$$

We are going to apply Sturm's Theorem \cite{Prasolov} to $P(x)$.
$$P_0(x)=P(x),\, P_1(x)=P'(x)=4A_1x^3+3B_1x^2+2C_1x+D_1.$$
Now, we perform Euclidean division to find $P_2(x)$. Divide $P_0(x)$ by $P_1(x)$.
$$A_1x^4+B_1x^3+C_1x^2+D_1x+E_1=\left(\frac{1}{4}x+\frac{B_1}{16A_1}\right)\left(4A_1x^3+3B_1x^2+2C_1x+D_1\right)+$$
$$+\left(\frac{C_1}{2}-\frac{3B_1^2}{16A_1}\right)x^2+\left(\frac{3}{4}D_1-\frac{B_1C_1}{8A_1}\right)x+\left(E_1-\frac{B_1D_1}{16A_1}\right).$$
$$P_2(x)=\left(\frac{3B_1^2}{16A_1}-\frac{C_1}{2}\right)x^2+\left(\frac{B_1C_1}{8A_1}-\frac{3}{4}D_1\right)x+\left(\frac{B_1D_1}{16A_1}-E_1\right).$$
We find remainder and $P_3(x)$, dividing $P_1(x)$ by $P_2(x)$ as follows
$$P_3(x)=\frac{16A_1}{(8A_1C_1-3B_1^2)^2}\cdot$$$$[(2 B_1^2 C_1^2 - 8 A_1 C_1^3 - 6 B_1^3 D_1 + 28 A_1 B_1 C_1 D_1 -36 A_1^2 D_1^2 -
   12 A_1 B_1^2 E_1 + 32 A_1^2 C_1 E_1)x$$ $$+(B_1^2C_1 D_1 - 4 A_1 C_1^2 D_1 + 3A_1B_1D_1^2 - 9B_1^3E_1 +32A_1 B_1 C_1 E_1 -
   48 A_1^2D_1 E_1)].$$
We find remainder and $P_4(x)$, dividing $P_2(x)$ by $P_3(x)$ as follows
$$P_4(x)=\frac{1}{64A_1}\cdot$$
$$\left(\frac{8A_1C_1-3B_1^2}{B_1^2C_1^2+4A_1C_1^3 + 3B_1^3D_1-14A_1B_1C_1D_1 + 18A_1^2D_1^2 + 6A_1B_1^2 E_1 - 16 A_1^2C_1E_1}\right)^2\cdot$$
$$\cdot(B_1^2C_1^2 D_1^2 -4A_1C_1^3D_1^2 - 4B_1^3D_1^3 + 18A_1B_1C_1D_1^3 -$$
$$- 27 A_1^2 D_1^4 -4B_1^2C_1^3 E_1 + 16 A_1 C_1^4 E_1 + 18 B_1^3C_1D_1E_1 -$$
$$- 80 A_1B_1C_1^2D_1E_1-6 A_1 B_1^2 D_1^2 E_1 + 144A_1^2C_1 D_1^2E_1 - 27B_1^4E_1^2 + 144A_1B_1^2C_1E_1^2 -$$
$$ -128 A_1^2 C_1^2 E_1^2 - 192 A_1^2B_1D_1 E_1^2 + 256A_1^3E_1^3).$$
Now we find signs of the $\left(P_0(x_0),P_1(x_0),P_2(x_0),P_3(x_0),P_4(x_0)\right)$, where $x_0\in\{0,+\infty\}$.
Note that $P_i(+\infty):=\lim\limits_{x\to +\infty}P_i(x)$ ($i=\overline{0,4}$) and sign of $P_i(+\infty)$ equals sign of leading coefficient of $P_i(x)$.
Let $\omega(x_0)$ be the number of sign changes in the sequence $$P_0(x_0),\,P_1(x_0),\,P_2(x_0),\,P_3(x_0),\,P_4(x_0)$$
and $\nu(0,+\infty)$ be the number of positive solutions of \eqref{aa31}. According to Sturm's Theorem $$\nu(0,+\infty)=\omega(0)-\omega(+\infty).$$
Through cumbersome calculations, we obtain the following schema for the sign of $P_i(x_0)$ (see Schema 1).
\begin{figure}[h!]
\includegraphics[width=12cm]{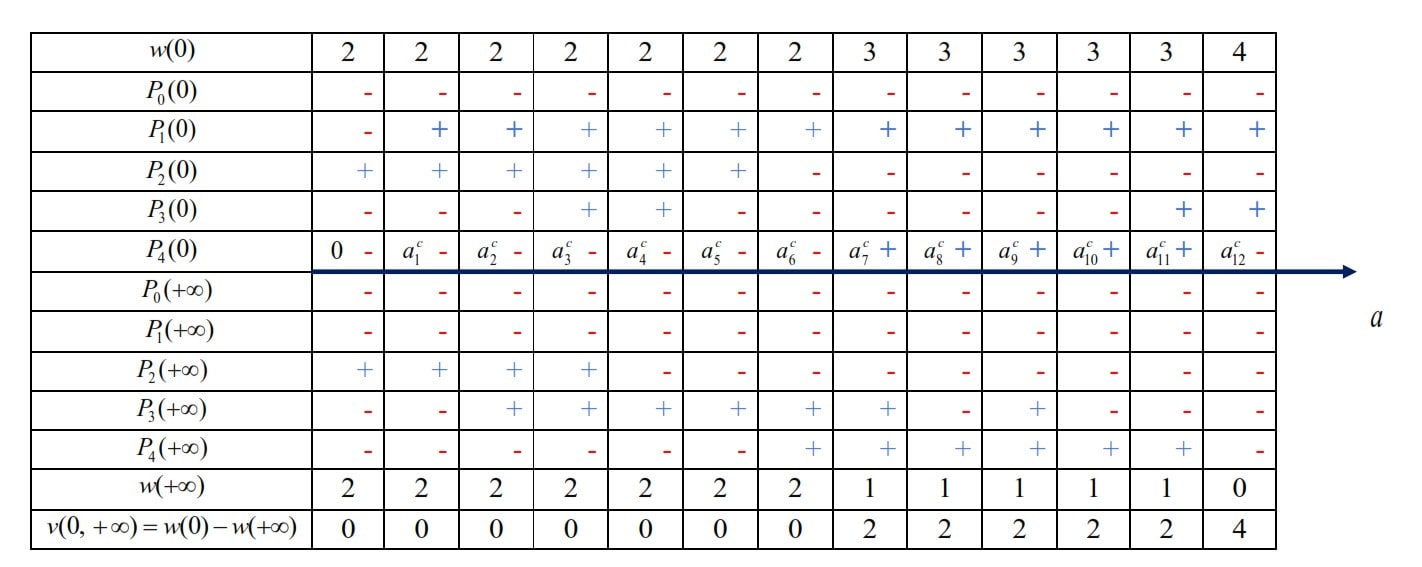}\label{fig1}
\end{figure}
\begin{center}{\footnotesize \noindent
 Schema 1. This schema illustrates the application of Sturm's Theorem to $P(x)$,

 where $a_1^c\approx0,215$, $a_2^c\approx0,592$, $a_3^c\approx1,690$, $a_4^c\approx1,929$, $a_5^c\approx1,947$, $a_6^c\approx1,960$, $a_7^c\approx2,010$, $a_8^c\approx2,035$, $a_9^c\approx2,226$, $a_{10}^c\approx3,609$, $a_{11}^c\approx3,717$, $a_{12}^c\approx4,921$.}
\end{center}

Let us prove statements in schema above for $P_2(x)$. Substituting parametric expressions of the coefficients into  $$P_2(0)=\frac{B_1D_1}{16A_1}-E_1,$$ we get
$$P_2(0)=-\frac{1}{32}(3a^9+16a^8+30a^7+16a^6-52a^5-224a^4-334a^3-352a^2-159a-32).$$
We consider the polynomial $$g_1(a):=3a^9+16a^8+30a^7+16a^6-52a^5-224a^4-334a^3-352a^2-159a-32.$$ By the Descartes Rule \cite{Prasolov}, $g_1(a)$ has at most one positive solution. It is easy to check that $g_1(1)=-1088<0$ and $g_1(2)=818>0$. There exist $a_6^c\in(1,2)$ ($a_6^c\approx1,96077$) such that $g_1(a_6^c)=0$. It is easy to check that if $a\in (0, a_6^c)$ then $P_2(0)>0$ and if $a\in (a_6^c, +\infty)$ then $P_2(0)<0$.

Clear that the sign of $P_2(+\infty)$ equals sign of $\frac{3B_1^2}{16A_1}-\frac{C_1}{2}$. Substituting parametric expressions of the coefficients into the last expression, we have $$-\frac{1}{16}(a+1)^2(3a^8+12a^7+12a^6-28a^5-70a^4-60a^3-36a^2-52a-5).$$
We consider the polynomial $$g_2(a):=3a^8+12a^7+12a^6-28a^5-70a^4-60a^3-36a^2-52a-5.$$ By the Descartes Rule, $g_2(a)$ has at most one positive solution. It is easy to check that $g_2(1)=-224<0$ and $g_2(2)=323>0$. There exist $a_4^c\in(1,2)$ ($a_4^c\approx1,929256$) such that $g_2(a_4^c)=0$. It is easy to check that if $a\in (0, a_4^c)$ then the sign of $P_2(+\infty)$ is positive and if $a\in (a_4^c, +\infty)$ then the sign of $P_2(+\infty)$ is negative.

The rest of the proof runs as before.

Due to schema and after some algebras, we conclude that the equation \eqref{aa31} does not have any positive solution if $a\in (0, a_7^c]$ (we denote $A_1^c:=a_7^c$), the equation \eqref{aa31} has two positive solutions if $a\in (a_7^c, a_{12}^c]$ (we denote $A_2^c:=a_{12}^c$), four positive solutions if $a\in (a_{12}^c,+\infty)$.
\end{proof}

\begin{figure}[h!]
    \centering
    % Grafikdan OLDIN KELGAN MATN (Grafiklardan uzoqlashtirmaslik uchun)
    \caption*{In the following figures illustrate graphics of $P(x)$ for $a=1.5$, $a=3$ and $a=5$ and clarify that the statements of Lemma \ref{numsol} are true.}

    % Birinchi grafik uchun minipage
    \begin{minipage}[b]{0.38\textwidth}
        \centering
        \includegraphics[width=\textwidth]{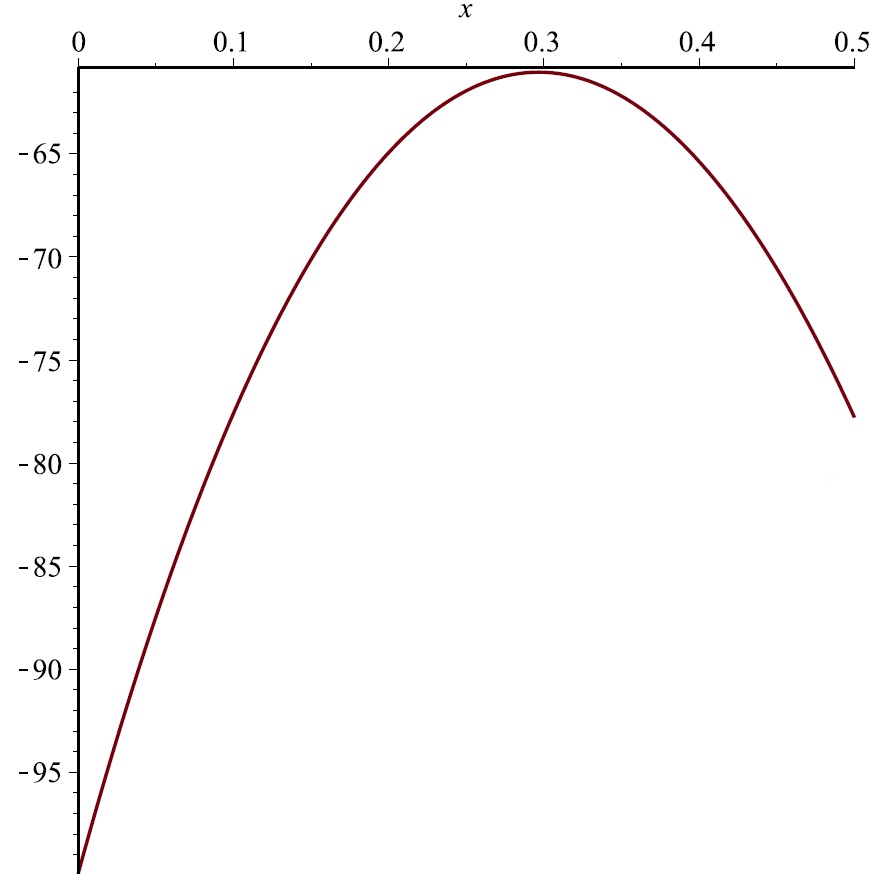}
        \subcaption{ $a=1.5$, $x\in(0, 0.5)$}
        \label{fig:grafik1}
    \end{minipage}
    \hfill % Ikki minipage o'rtasida bo'sh joy qoldirish
    % Ikkinchi grafik uchun minipage
    \begin{minipage}[b]{0.38\textwidth}
        \centering
        \includegraphics[width=\textwidth]{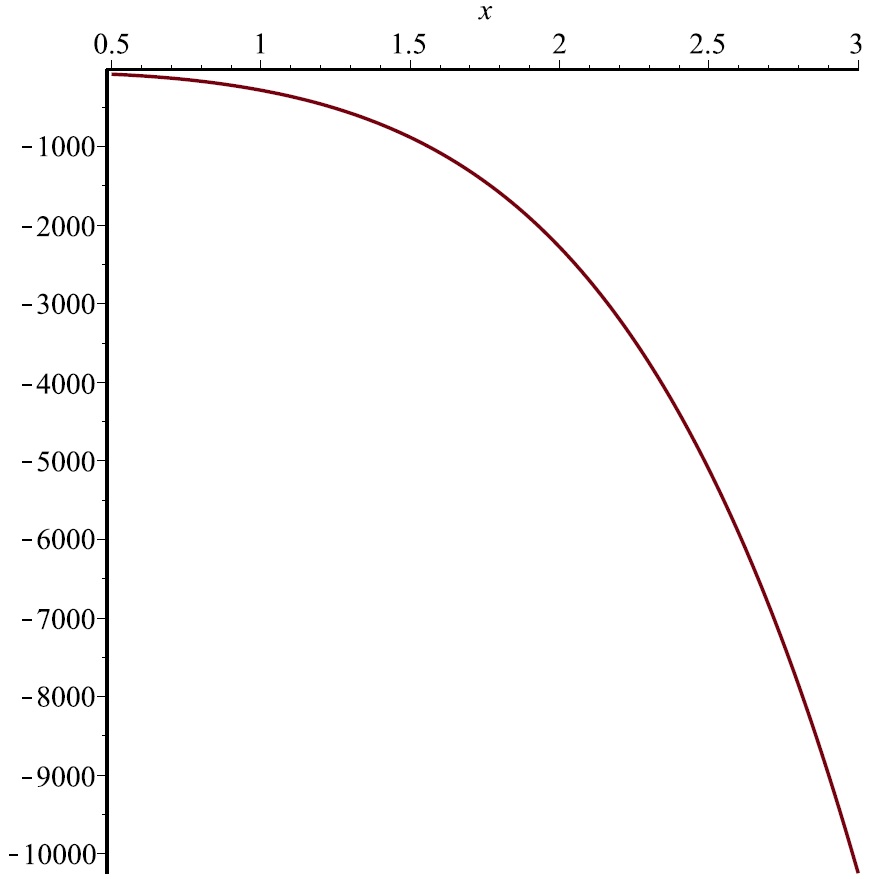}
        \subcaption{$a=1.5$, $x\in(0.5, 3)$}
        \label{fig:grafik2}
    \end{minipage}
    % Grafiklardan KEYIN KELGAN MATN (umumiy izohdan keyin qo'shilishi mumkin)
    \caption*{Figure 1. This figures demonstrate graphics of $P(x)$ for $a=1.5$ in intervals $x\in(0, 0.5)$ and $x\in(0.5,3)$, respectively.}
    \label{fig:ikkita-grafik}
\end{figure}

\begin{figure}
    \centering
    % Grafikdan OLDIN KELGAN MATN (Grafiklardan uzoqlashtirmaslik uchun)
    \caption*{Computer programm shows that $P_{max}(x)\approx-61$ for $a=1.5$ in $(0,+\infty)$, i.e. in this case, the equation \eqref{aa31} does not have any positive solution.}

    % Birinchi grafik uchun minipage
    \begin{minipage}[b]{0.49\textwidth}
        \centering
        \includegraphics[width=\textwidth]{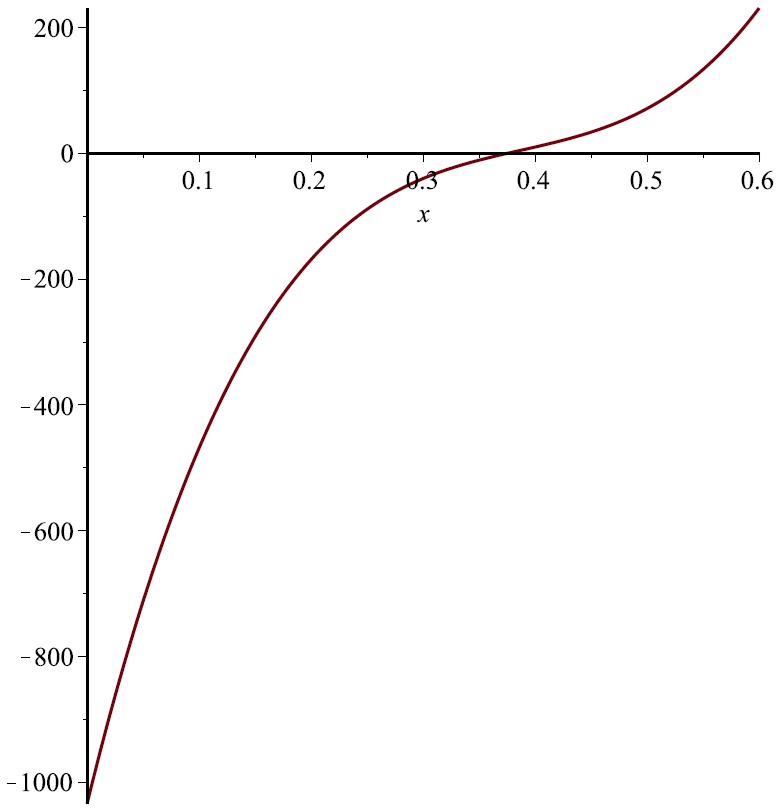}
        \subcaption{ $a=3$, $x\in(0, 0.5)$}
        \label{fig:grafik1}
    \end{minipage}
    \hfill % Ikki minipage o'rtasida bo'sh joy qoldirish
    % Ikkinchi grafik uchun minipage
    \begin{minipage}[b]{0.49\textwidth}
        \centering
        \includegraphics[width=\textwidth]{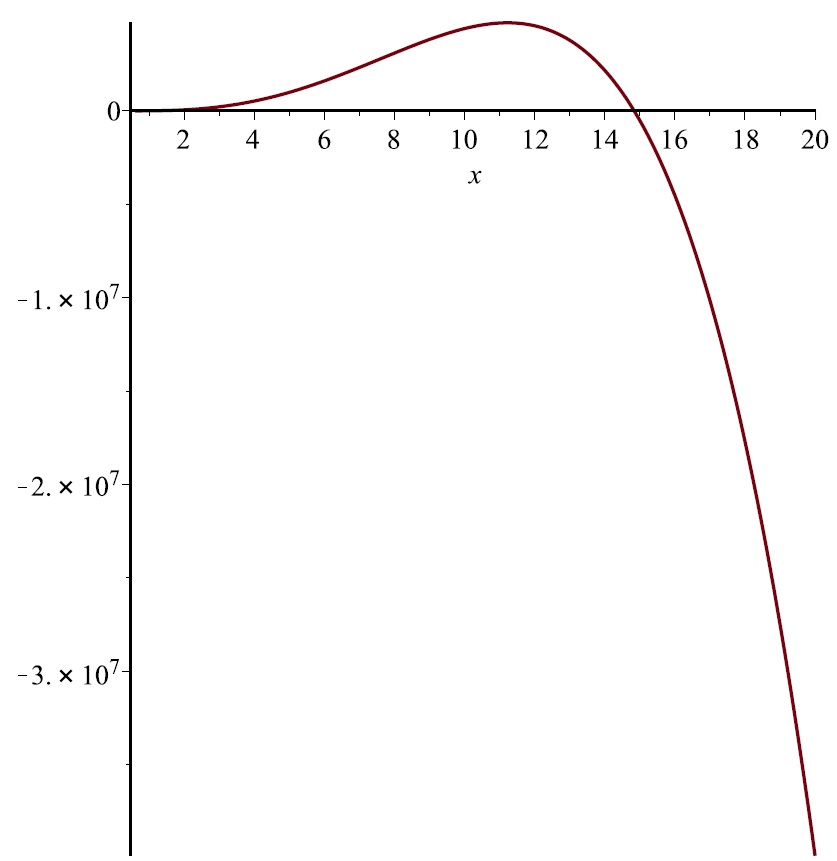}
        \subcaption{$a=3$, $x\in(0.5, 3)$}
        \label{fig:grafik2}
    \end{minipage}
    % Grafiklardan KEYIN KELGAN MATN (umumiy izohdan keyin qo'shilishi mumkin)
    \caption*{Figure 2. This figures demonstrate graphics of $P(x)$ for $a=3$ in intervals $x\in(0, 0.6)$ and $x\in(0.6,20)$, respectively.}
    \label{fig:ikkita-grafik}
\end{figure}

\begin{figure}
    \centering
    % Grafikdan OLDIN KELGAN MATN (Grafiklardan uzoqlashtirmaslik uchun)
    \caption*{In this case, the equation \eqref{aa31} has two positive solutions.
}
    % Birinchi grafik uchun minipage
    \begin{minipage}[b]{0.49\textwidth}
        \centering
        \includegraphics[width=\textwidth]{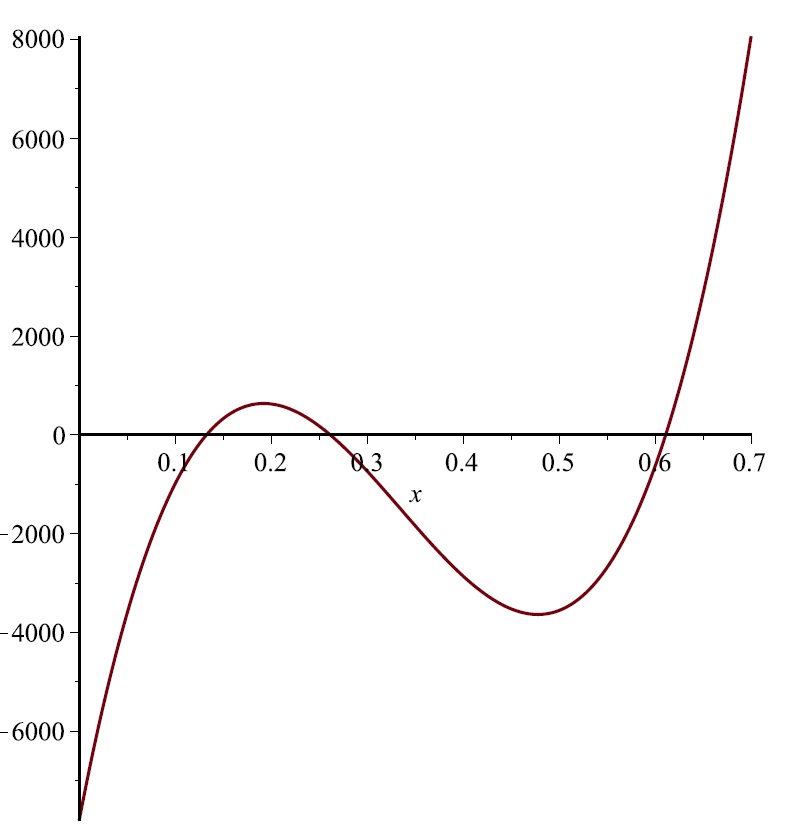}
        \subcaption{ $a=5$, $x\in(0, 0.7)$}
        \label{fig:grafik1}
    \end{minipage}
    \hfill % Ikki minipage o'rtasida bo'sh joy qoldirish
    % Ikkinchi grafik uchun minipage
    \begin{minipage}[b]{0.49\textwidth}
        \centering
        \includegraphics[width=\textwidth]{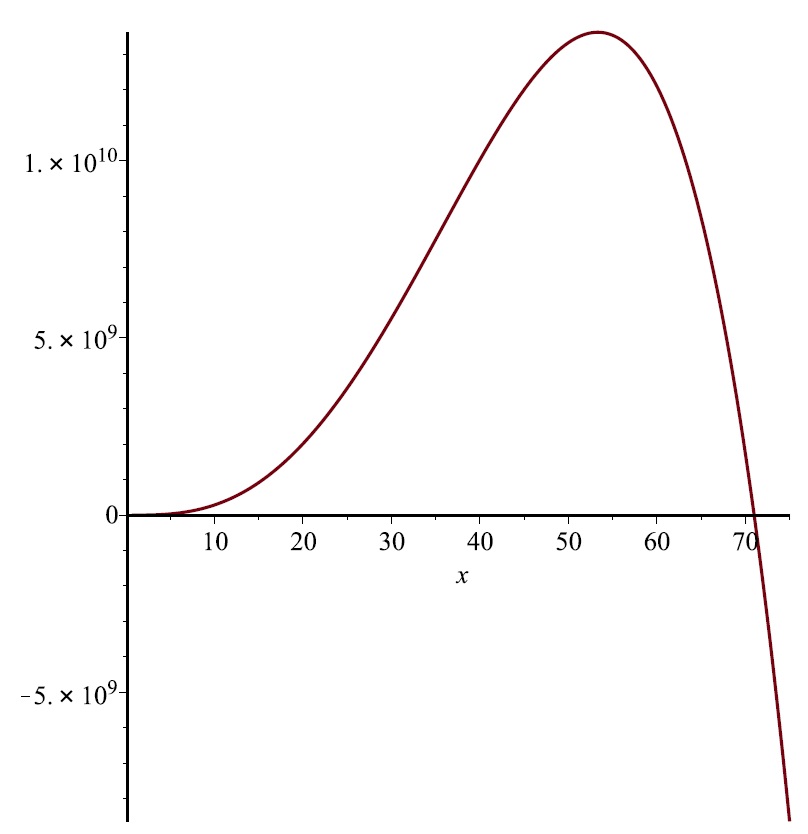}
        \subcaption{$a=5$, $x\in(0.7, 75)$}
        \label{fig:grafik2}
    \end{minipage}
    % Grafiklardan KEYIN KELGAN MATN (umumiy izohdan keyin qo'shilishi mumkin)
    \caption*{Figure 3. This figures demonstrate graphics of $P(x)$ for $a=5$ in intervals $x\in(0, 0.7)$ and $x\in(0.7,75)$, respectively.}
In this case, the equation \eqref{aa31} has four positive solutions.
\end{figure}

\begin{theorem}\label{k2ti3}
Let $k=2$, $J_P=\frac{\alpha}{1-\alpha}J_I$ and $a=\exp\{J\beta\alpha\}$,  $a\neq1$, $\alpha\in(0,1)$. For the translation- invariant Gibbs measures of the $(2,3)$-Ising-Potts model, corresponding to the solutions of \eqref{funceqti} on the invariant set $I_2$, the following assertions hold
\begin{itemize}
  \item If $a\in (0,1)\cup (1, B^c_1)$ then there are 2 TIGMs;
  \item If $a\in [B^c_1,B^c_2)$ then there are 4 TIGMs;
  \item If $a=B^c_2$ then there are 3 TIGMs;
  \item If $a\in (B^c_2,B^c_3]$ then there are 4 TIGMs;
  \item If $a\in (B^c_3,B^c_4)$ then there are 6 TIGMs;
  \item If $a=B^c_4$ then there are 5 TIGMs;
  \item If $a\in (B^c_4,B^c_5)$ then there are 6 TIGMs;
  \item If $a\in [B^c_5,+\infty)$ then there are 8 TIGMs.
\end{itemize}
where $B^c_1\approx2,010\, \ B^c_2\approx2,179\, \ B^c_3\approx3,828\, \ B^c_4\approx4,071\, \ B^c_5\approx4,921 $.
\end{theorem}
\begin{proof} Let $a\in (0,1)\cup (1, B^c_1)$ then there exist only $u_0^{(1)}=1$ and $u_0^{(3)}=1$.

The rest of the proof runs as before.
\end{proof}
\begin{remark}\label{comp}
\begin{itemize}
  \item In \cite{KRH14}, under the certain conditions existence of $2^q-1$ TIGMs for the Potts model on a Cayley tree of order $k\geq2$. Theorem \ref{k2ti3} shows that if $a\in [B^c_5,+\infty)$  then there exist at least $8$ (corresponding to the invariant $I_3$) TIGMs for the $(2,3)$-state Ising-Potts model on the Cayley tree of order two, i.e. we have established more measures than previously known.
  \item In \cite{HOR25}, under the certain conditions existence of 335 TIGMs and 12 critical temperatures for the 5-state Ising-Potts model (different from \eqref{Ham}) on a Cayley tree of order two. Theorem \ref{k2ti3} shows that for the $(2,3)$-state Ising-Potts model (defined by \eqref{Ham}) under the certain conditions, existence of at least 8 TIGMs and 5 critical temperatures  on the Cayley tree of order two.
  \end{itemize}
\end{remark}

\section{Discussion: Open problems}
We now formulate some problems whose solution turns out to be sufficiently difficult, and they require further consideration:
\begin{enumerate}
  \item How many solution of \eqref{sys i24} are there if $a\neq b$?
  \item How many solution of \eqref{funceqti} does have outside the invariant set $I_3$?
  \item Finding extremality conditions of the found Gibbs measures?
  \item Can we find periodic or weakly periodic Gibbs measures for this model?
  \item How can we conclude about the dynamics of the mapping \eqref{funceqtiop}?
\end{enumerate}

\section{Acknowledgements}
The authors would like to thank professors Nasir N. Ganikhodjaev and U.A. Rozikov for suggesting this problem.

\section{Funding}
No funds, grants, or other support was received.

\section{Data availability statement}
Not applicable

\section{Conflicts of interest}
The author declares that they have no conflict of interest.

\end{document}